\documentclass[a4paper,12pt]{article}
\usepackage{amsmath}
\usepackage{amsthm}
\usepackage{amssymb}
\usepackage{amscd}
\usepackage{graphicx}
\usepackage{epsfig}
\usepackage[matrix,arrow,curve]{xy}
\usepackage{color}
\usepackage{mathrsfs}
\usepackage[scr=rsfs,cal=boondox]{mathalpha}

\usepackage{bbm}
\usepackage{stmaryrd}
\usepackage{hyperref}
\usepackage{multirow}

\usepackage{latexsym}
\usepackage{amsfonts}
\input xy
\usepackage{tikz}
\usetikzlibrary{matrix}
\usetikzlibrary{decorations.pathreplacing,angles,quotes}

\newtheorem{theorem}{Theorem}[section]
\newtheorem{proposition}[theorem]{Proposition}

\newtheorem{corollary}[theorem]{Corollary}

\theoremstyle{definition}
\newtheorem{example}[theorem]{Example}
\newtheorem{remark}[theorem]{Remark}
\newtheorem{definition}[theorem]{Definition}

\mathchardef\za="710B  
\mathchardef\zb="710C  
\mathchardef\zg="710D  
\mathchardef\zd="710E  
\mathchardef\zve="710F 
\mathchardef\zz="7110  
\mathchardef\zh="7111  
\mathchardef\zvy="7112 
\mathchardef\zi="7113  
\mathchardef\zk="7114  
\mathchardef\zl="7115  
\mathchardef\zm="7116  
\mathchardef\zn="7117  
\mathchardef\zx="7118  
\mathchardef\zp="7119  
\mathchardef\zr="711A  
\mathchardef\zs="711B  
\mathchardef\zt="711C  
\mathchardef\zu="711D  
\mathchardef\zvf="711E 
\mathchardef\zq="711F  
\mathchardef\zc="7120  
\mathchardef\zw="7121  
\mathchardef\ze="7122  
\mathchardef\zy="7123  
\mathchardef\zf="7124  
\mathchardef\zvr="7125 
\mathchardef\zvs="7126 
\mathchardef\zf="7127  
\mathchardef\zG="7000  
\mathchardef\zD="7001  
\mathchardef\zY="7002  
\mathchardef\zL="7003  
\mathchardef\zX="7004  
\mathchardef\zP="7005  
\mathchardef\zS="7006  
\mathchardef\zU="7007  
\mathchardef\zF="7008  
\mathchardef\zW="700A  

\newcommand{\be}{\begin{equation}}
\newcommand{\ee}{\end{equation}}

\newcommand{\bea}{\begin{eqnarray}}
\newcommand{\eea}{\end{eqnarray}}
\newcommand{\beas}{\begin{eqnarray*}}
\newcommand{\eeas}{\end{eqnarray*}}
\def\*{{\textstyle *}}
\newcommand{\T}{{\mathbb T}}

\newcommand{\we}{\wedge}

\newcommand{\Z}{\mathbb{Z}}
\newcommand{\R}{\mathbb{R}}

\newcommand{\C}{\mathbb{C}}
\newcommand{\Q}{\mathbb{Q}}

\newcommand{\pa}{\partial}
\newcommand{\ti}{\times}

\newcommand{\Ll}{{\pounds}}

\def\tU{{\widetilde{U}}}
\def\tM{\widetilde{M}}
\def\tw{\widetilde{\zw}}
\def\tR{{{\cR_\tU}}}
\def\th{{\zh_\tU}}

\def\cR{{\mathcal R}}

\def\cO{{\mathcal O}}

\def\cU{{\mathcal U}}

\def\Exp{\operatorname{Exp}}

\def\sT{{\mathsf T}}

\def\xd{\mathrm{d}}
\def\xi{\mathrm{i}}

\def\cF{{\mathcal F}}

\newdir{|>}{%
!/4.5pt/@{|}*:(1,-.2)@^{>}*:(1,+.2)@_{>}}

\newdir{ (}{{}*!/-5pt/@^{(}}

\newcommand{\id}{\mathrm{id}}

\def\n{\nabla}

\newcommand{\m}{{\medskip}}
\newcommand{\mn}{{\medskip\noindent}}

\newcommand{\no}{{\noindent}}
\def\Rt{{\R^\ti}}

\begin{document}
\title{Regular contact manifolds:\\ a generalization of the Boothby-Wang theorem}
\author{Katarzyna Grabowska\footnote{email:konieczn@fuw.edu.pl }\\
\textit{Faculty of Physics,
                University of Warsaw}
\\ \\
Janusz Grabowski\footnote{email: jagrab@impan.pl} \\
\textit{Institute of Mathematics, Polish Academy of Sciences}}
\date{}
\maketitle
\begin{abstract} A \emph{regular contact manifold} is a manifold $M$ equipped with a globally defined contact form $\zh$ such that the topological space $M/\cR$ of orbits (trajectories) of the Reeb vector field $\cR$ of $\zh$ carries a smooth manifold structure, so the canonical projection $p:M\to M/\cR$ is a smooth fibration. We show that, under the additional assumption that $\cR$ is a complete vector field, this fibration is actually either an $S^1$- or an $\R$-principal bundle. Moreover, there exists a unique symplectic form $\zw$ on $M/\cR$ such that $p^*(\zw)=\xd\zh$ which is $\zr$-integral in the $S^1$-bundle case, where $\zr$ is the minimal period of the $S^1$-action, so the symplectic manifold  $(M/\cR,\zw)$ admits a prequantization. We do not assume that $M$ is compact.

\medskip\noindent
{\bf Keywords:}
\emph{contact form; Reeb vector field; smooth fibration; fiber bundle; principal bundle; symplectic form; integrality condition; prequantization.}\par

\smallskip\noindent
{\bf MSC 2020:} 53D10; 53D35; 37C10; 37C86.	

\end{abstract}
\section{Introduction}
The main object of our studies in this paper is the structure of regular contact manifolds.
More precisely, let $\zh$ be a contact form on a manifold $M$, which will be assumed to be connected throughout this paper. We say that the contact manifold $(M,\zh)$ is \emph{regular} if the foliation $\cF_\cR$ of $M$ by orbits of the Reeb vector field $\cR$ of $\zh$ is simple, i.e., the space $M/\cF_\cR$ of orbits has a manifold structure such that the canonical projection $p:M\to M_\cR$ is a smooth fibration. Here, by orbits of $\cR$ (which is a nonvanishing vector field on $M$) we understand the 1-dimensional submanifolds of $M$, being the images of trajectories of $\cR$. Since any orbit does not depend on the parametrization, $\cF_\cR=\cF_{f\cR}$ for a nonvanishing function $f:M\to\R$.

\m The structure of compact regular contact manifolds $(M,\zh)$ has been studied in \cite{Boothby:1958}. Theorem 1 there says that on such a manifold there exists an equivalent contact form $\zh'$ whose Reeb vector field $\cR'$ is periodic, thus inducing a principal action of the group $S^1=\R/\Z$ on $M$. However, the proof in \cite{Boothby:1958} is incomplete in one important respect. It has been already noticed and corrected in \cite{Geiges:2008,Niederkruger:2005}. This proof starts with the observation that, as orbits are closed submanifolds in the regular case, for compact $M$ they all are circles, so $\cR$ is periodic on each orbit, however, \emph{a priori} with different periods. It has been proved that one can find a nonvanishing smooth function $f:M\to\R$ such that all trajectories of $\cR'=f\cR$ have a common minimal period, so $\cR'$ induces a free $S^1$-action on $M$. Further, the authors define this equivalent contact form to be $\zh'=\zh/f$, and claim that $\cR'$ is the Reeb vector field of $\zh'$, which is clearly false, except when $f$ is a constant.

\m In this paper we show, in particular, that \cite[Theorem 1]{Boothby:1958} is actually true and in a stronger version: the Reeb vector field $\cR$ itself is automatically the fundamental vector field of a free $S^1$-action, so we do not need to seek for a rescaling of $\cR$. What is more, our main result generalizes the Boothby-Wang theorem, as it includes arbitrary (not only compact) regular contact manifolds for which the Reeb vector field is complete (this is automatically satisfied for compact $M$). It is clear that \emph{a priori} $\cR$ may have both, compact orbits as well as non-compact ones. We show, however, that for complete $\cR$ we have only two possibilities: either $\cR$ has no periodic orbits, or all orbits share the same minimal period.
\begin{theorem}
If the Reeb vector field $\cR$ on a regular contact manifold $(M,\zh)$ is complete, then it induces on $M$ either an $\R$- or an $S^1$-principal action, so $p:M\to M/\cR$ is a principal bundle. Moreover, there exists a symplectic form $\zw$ on $M/\cR$ such that $p^*(\zw)=\xd\zh$ and, in the first case, $\zw$ is exact, while in the second case the symplectic form $\zw$ is $\zr$-integral, where $\zr$ is the minimal period of the $S^1$-action.
\end{theorem}
\no It is easy to see that, in the case of an $S^1$-action, $(M,\zh)$ induces a prequantization of the symplectic manifold $(M/\cR,\zw)$. Note also that if $\cR$ indices a non-free $S^1$-action, then the quotient manifold $M/S^1$ is generally only an orbifold \cite{Kegel:2021}.

\section{All equivalent contact forms in one picture}
Generally, a \emph{contact structure} on a manifold $M$ of dimension $(2n+1)$ is a \emph{maximally nonintegrable} distribution $C\subset \sT M$, being a \emph{field of hyperplanes} on $M$, i.e., a distribution with rank $2n$. Such a distribution is, at least locally, the kernel of a nonvanishing 1-form $\zh$ on $M$, i.e., $C=\ker(\zh)$. Of course, the 1-form $\zh$ is determined only up to conformal equivalence. Such a (local) 1-form we call a \emph{contact form}. It is characterized by the condition that $\zh\we(\xd\zh)^n$ is nonvanishing, i.e., is a volume form. In this paper we will consider only \emph{trivial} (\emph{co-oriented}) contact structures, i.e., manifolds equipped with a globally defined contact form. The local picture for contact forms is fully described by the following.
\begin{theorem}[Contact Darboux Theorem] Let $\zh$ be 1-form on a manifold $M$ of dimension $(2n+1)$. Then $\zh$ is a contact form if and only if around every point of $M$ there are local coordinates $(z,p_i,q^i)$, $i=1,\dots,n$, in which $\zh$ reads
\be\label{Dc} \zh=\xd z-p_i\,\xd q^i.\ee
\end{theorem}
\no It is now easy to see that, for any nonvanishing function $f$ on $M$, a 1-form $\zh$ is a contact form if and only if $f\zh$ is a contact 1-form. The contact form $f\zh$ we call \emph{equivalent} with $\zh$. It defines the same contact distribution $C=\ker(\zh)$.

\no Any contact form $\zh$ on $M$ determines uniquely a nonvanishing vector field $\cR$ on $M$, called the \emph{Reeb vector field}, which is characterized by the equations
$$i_\cR\zh=1\quad \text{and}\quad i_\cR\xd\zh=0.$$
The Reeb vector field for the contact form (\ref{Dc}) is $\cR=\pa_z$.

\m Now, for a cooriented contact manifold $(M,\zh)$ let us consider its symplectization understood as the manifold $\tM=M\ti\R_+$ equipped with the 1-homogeneous symplectic form
$$\tw(x,s)=\xd(s\cdot\zh)(x,s)=\xd s\we\zh(x)+s\cdot\xd\zh(x).$$
Here, $\R_+$ is the multiplicative group of positive integers and the homogeneity of degree $a\in\R$ of a differential form $\zb$ on $\tM$ means that
\be\label{ho}h_s^*(\zb)=s^a\cdot\zb,\ee
where $h_s(x,s_0)=(x,ss_0)$ is the canonical principal action of $\R_+$ on $\tM$. In particular, $M=\tM/\R_+$, so
$$\zt:\tM\ni(x,s)\mapsto x\in M=\tM/\R_+$$ is the trivial $\R_+$-principal bundle. The generator of the $\R_+$-action is the vector field $\n=s\,\pa_s$ which we will call the \emph{Euler vector field}.
Note that (\ref{ho}) is equivalent to $\Ll_\n(\zb)=a\cdot\zb$. The 1-form $$\zvy(x,s)=(i_\n\tw)=s\cdot\zh,$$
which is the only 1-homogeneous semibasic potential for $\tw$, $\xd\zvy=\tw$, we call the \emph{Liouville form}.

The homogeneity of $\tw$ implies that the Hamiltonian vector field $X_H$ of any 1-homogeneous Hamiltonian $H$ on $\tM$ is $\R_+$-invariant, and therefore projects to a vector field $X^c_H$ on $M$. Since 1-homogeneous Hamiltonians are of the form $H(x,s)=s\cdot G(x)$, we will denote $X^c_{sG}$ simply by $X_G$. In contact mechanics, $X_G$ is called the \emph{contact Hamiltonian vector field} with the Hamiltonian $G$. It is easy to show (cf. \cite{Grabowska:2022}) that $X_G$ is uniquely determined by the equations
$$
i_{X_G}\eta =G,\qquad i_{X_G}\xd\eta =\mathcal{R}({G})\eta-\xd {G}\,,
$$
where $\cR$ is the Reeb vector field of $\zh$. It is indeed a contact vector field, since
$$\Ll_{X_G}\zh=\xd(i_{X_G}\zh)+i_{X_G}\xd\zh=\xd G+\mathcal{R}({G})\eta-\xd G=\mathcal{R}({G}) \eta.$$

\mn Any function $F:M\to\R_+$ defines a section
$$\zs_F:M\to\tM,\quad \zs_F(x)=(x,F(x))\in\tM,$$
of the principal bundle $\zt:\tM\to M$ whose image is $M_F=\{(x,F(x)):\,x\in M\}$. In other words, $M_F$ is defined by the equation $H(x,s)=1$, where $H$ is the 1-homogeneous Hamiltonian $H(x,s)=s/F(x)$ on $\tM$. It is obvious that $\zs_F$ is a diffeomorphism of $M$ onto $M_F$. The canonical contact form on $M_F$ is the restriction $\zh_F$ of the Liouville form $\zvy$ to $M_F$, and $\zs_F^*(\zh_F)=F\zh$ is a contact form equivalent to $\zh$. Conversely, any 1-form  equivalent to $\zh$ can be obtained in this way for some $F$. In other words, all contact forms on $M$ which are equivalent to $\zh$ are in this one-to-one correspondence with sections of the principal bundle $\zt:\tM\to M$. Since $M_F$ is of codimension 1 in $\tM$, it is a coisotropic submanifold whose characteristic foliation consists of orbits of the Hamiltonian vector field $X_H$ of the 1-homogeneous Hamiltonian $H(x,s)=s/F(x)$. The restriction of the symplectic form $\tw=\xd(s\zh)$ to $M_F$ is $\xd(F\zh)$, so the projection $X$ of $X_H$ to $M$ satisfies $i_X\xd(F\zh)=0$. Moreover, because $H$ is 1-homogeneous, we have $\Ll_\n(\xd H)=\xd H$, so
$$\xd(i_{X_H}\zvy)=\xd(i_{X_H}i_\n\tw)=-\xd(i_\n i_{X_H}\tw)=\xd(i_\n\xd H)=\Ll_\n(\xd H)=\xd H.$$
Since both, $i_{X_H}\zvy$ and $\xd H$ are 1-homogeneous, it follows that $H=i_{X_H}\zvy$.
Hence, $i_{X_H}\zvy=1$ on $M_F$, thus $i_X(F\zh)=1$. Summing up, we get that $X$ is the Reeb vector field of the contact form $F\zh$. Note that the submanifolds $M_F$ of $\tM$ are submanifolds \emph{of contact type} in the terminology of Weinstein \cite{Weinstein:1979}.

\begin{proposition}\label{reeb1}
For $F:M\to\R_+$, the Reeb vector field $\cR_F$ of the contact form $\zh_F=F\zh$ on $M$ is the Hamiltonian contact vector field associated with the contact Hamiltonian $1/F$. In other words,
$\cR_F=\zt_*(X_H)$, where $X_H$ is the Hamiltonian vector field on $\tM$ associated with the 1-homogeneous Hamiltonian $H(x,s)=s/F(x)$.
\end{proposition}
\begin{remark} The picture presented above is a particular case of a \emph{symplectic $\Rt$-principal bundle} in the terminology of \cite{Grabowski:2013} (here, $\Rt=\R\setminus\{0\}$ is the multiplicative group of nonzero reals). Such bundles are defined as $\Rt$-principal bundles $\zt:\tM\to M$ equipped additionally with a 1-homogeneous symplectic form $\tw$. The importance of these geometric objects comes from the fact that they canonically induce on $M$ a contact structure, not necessarily cooriented. Cooriented contact structures correspond to trivial $\Rt$-principal bundles
$\zt:\tM=M\ti\Rt\to M$ which, for connected $M$, consists of two connected components, and in this case it is enough to consider only the component $M\ti\R_+$. If the principal bundle $\tM$ is not trivializable, it is connected for connected $M$. The necessity of using the non-connected group $\Rt$ instead of just $\R_+$ comes from the requirement of including non-cooriented contact structures into the picture. Closer studies of such structures, together with the corresponding contact Hamiltonian systems, one can find in a series of papers \cite{Bruce:2017,Grabowska:2022,Grabowska:2022a,Grabowski:2013}.
\end{remark}

\mn
The mistake in \cite{Boothby:1958} was the false claim that, for a nonvanishing function $f$ on $M$, the vector field $f\cR$ is the Reeb vector field for the contact form $\zh/f$.
Indeed, if $\cR'=f\cR$ is the Reeb vector field of $\zh'=F\zh$, then necessarily $F=1/f$. But
$$\xd\zh'=\xd(\zh/f)=\xd\zh/f-\big(\xd f/f^2\big)\we\zh,$$
and therefore
$$i_{\cR'}\xd\zh'=\xd f/f-\big(\cR(f)/f\big)\zh.$$
This is constantly 0 if and only if $\cR(f)\zh=\xd f$. Hence, for any vector field $X$ taking values in the contact distribution $C=\ker(\zh)$, we get
$$0=i_X\big(\cR(f)\zh\big)=X(f).$$
This implies that also $[X,Y](f)=0$ for all $X,Y\in C=\ker(\zh)$. But $C$ is maximally non-integrable, so
such Lie brackets span the whole tangent bundle $\sT M$, thus $f$ is a constant. Actually, the phase portraits of
the Reeb vector fields of equivalent contact forms can be drastically different.
\begin{example}\label{e1}
The unit sphere $M=S^3$ in $\R^4$ with coordinates $(q^1,p_1,q^2,p_2)$ carries a canonical contact form $\zh$ being the restriction of the Liouville 1-form
$$\zvy=q^1\xd p_1-p_1\xd q^1+q^2\xd p_2-p_2\xd q^2$$
to the sphere. Note that $\R^4$ is canonically a symplectic manifold with the symplectic form
$$\zw=2\big(\xd q^1\we\xd p_1+\xd q^2\we\xd p_2\big).$$ We can also view $M$ as the unit sphere in $\C^2$, where we identify $z_k=q^k+ip_k\in\C$ with $(q^k,p_k)\in\R^2$, $k=1,2$. The Hamiltonian vector field
$$X_H=\big(q^1\pa_{p_1}-p_1\pa_{q^1}\big)+\big(q^2\pa_{p_2}-p_2\pa_{q^2}\big)$$
with the Hamiltonian
$$H(z_1,z_2)=\frac{1}{4}\big(|z_1|^2+|z_2|^2\big)$$
is tangent to $S^3$ and represents there the Reeb vector field of $\zh$.

\mn For $a,b\ge 0$ consider a new Hamiltonian
$$H_{a,b}(z_1,z_2)=\frac{a}{4}|z_1|^2+\frac{b}{4}|z_2|^2.$$
Denote with $F_{a,b}$ the restriction of $H_{a,b}$ to $S^3:|z_1|^2+|z_2|^2=1$.
Of course, $H_{1,1}$ is our old Hamiltonian $H$. The projection along the rays of $\R^4$ onto $S^3$ maps the Hamiltonian vector field $X_{H_{a,b}}$ onto the Reeb vector field $\cR_{a,b}$ of the contact form $\zh/F_{a,b}$ on $S^3$. Consequently, trajectories of $X_{H_{a,b}}$ project onto trajectories of $\cR_{a,b}$.
It is easy to see that
$$
X_{H_{a,b}}=a\big(q^1\pa_{p_1}-p_1\pa_{q^1}\big)+b\big(q^2\pa_{p_2}-p_2\pa_{q^2}\big),
$$
so the trajectories of $X_{H_{a,b}}$ are of the form
$$\R\ni t\mapsto \big(e^{ait}z_1,e^{bit}z_2\big),$$
They project onto the trajectories
$$\R\ni t\mapsto \frac{1}{\sqrt{|z_1|^2+|z_2|^2}}\Big(e^{ait}z_1,e^{bit}z_2\Big)$$
of $\cR_{a,b}$. For every $0<r<1$, the trajectory starting from a point $(z_1,z_2)$ of the 2-dimensional torus
$$T_r=\{(z_1,z_2)\in S^3:\,|z_1|^2=r\}$$
lays entirely on this torus and is closed if and only if $a/b\in\Q$. Of course, this is the case of the original Reeb vector field $\cR=\cR_{1,1}$, but it is clear now that even an arbitrary close to 1 factor $1/F$ in $\zh/F$ will result in a radical qualitative change of the phase portrait of the corresponding Reeb vector field. We indicated only not closed orbits of $\cR_{a,b}$ for $a/b\notin\Q$, but according to the Weinstein Conjecture (which is true for $S^3$ \cite{Hofer:1993}) there must be a closed orbit of a point of $S^3$. Actually, there are two such orbits,
$$\R\ni t\mapsto \big(0,e^{bit}\big)\ \text{and}\ \R\ni t\mapsto \big(e^{2ait},0\big)$$
with the minimal periods $2\pi/b$ and $2\pi/a$, respectively.
\end{example}
\section{Contactizations}
It is obvious that any co-oriented contact manifold $(M,\zh)$ of dimension $(2n+1)$ is automatically presymplectic with the exact presymplectic form $\xd\zh$ of rank $2n$. In this case the involutive distribution $\ker(\xd\zh)$ is generated by the Reeb vector field $\cR$. In the following we will use the terminology of \cite{Boothby:1958}.
\begin{definition}
A cooriented contact manifold $(M,\zh)$ we call \emph{regular} if the foliation $\cF_\cR$ of $M$ by $\cR$-orbits is simple, i.e., the space $M/\cR=M/\cF_\cR$ of orbits of $\cR$ carries a smooth manifold structure such that the canonical projection $p:M\to M/\cR$ is a surjective submersion. In other words, $p$ is a smooth fibration. We call $(M,\zh)$ \emph{complete} if the Reeb vector field is complete.
\end{definition}
\begin{remark} Of course, any regular compact contact manifold is complete. The dynamics of the Reeb vector fields on compact contact manifolds is a subject of intensive studies, partly because its relation to Hamiltonian dynamics on a fixed energy hypersurface. For a very general geometric approach to contact Hamiltonian mechanics as a part of the classical symplectic Hamiltonian mechanics we refer to \cite{Grabowska:2022}.
A long-standing open problem concerning the Reeb dynamics is the so-called \emph{Weinstein Conjecture}, stating that for contact forms on compact manifolds the corresponding Reeb vector fields carry at least one periodic orbit. Note that in \cite{Weinstein:1979} it was supposed additionally that the manifold is simply connected, because the author suspected the existence of a counterexample for a torus. The counterexample appeared to be false and this assumption has been finally dropped. This conjecture has been proved for some particular cases, especially for 3-dimensional manifolds \cite{Taubes:2007}. The origins and the history of the Weinstein Conjecture are nicely described in \cite{Pasquotto:2012}.
\end{remark}
\no The contact manifolds considered by Weinstein were hypersurfaces $M$ in a symplectic manifold $(P,\zW)$, equipped with a contact form $\zh$ such that $\zW\,\big|_M=\xd\zh$, where $\zW\,\big|_M$ is the restriction of $\zW$ to $M$. Weinstein called them \emph{hypersurfaces of contact type}. Since any hypersurface in a symplectic manifold is automatically coisotropic, the corresponding Reeb vector field on a hypersurface of contact type spans its characteristic distribution. It is proved in \cite{Weinstein:1979} that any hypersurface of contact type can be obtained by a \emph{symplectic-to-contact reduction}.
\begin{proposition}[symplectic-to-contact reduction] Let $(P,\zW)$ be a symplectic manifold and $M$ be a hypersurface in $P$. If $\n$ is a vector field, defined in a neighbourhood of $M$, such that
$\n$ is transversal to $M$ and $\Ll_\n\zW=\zW$, then the restriction $\zh$ to $M$ of the 1-form
$\tilde\zh=i_\n\zW$ is a contact form on $M$, and $\xd\zh=\zW\,\big|_M$.
\end{proposition}
\begin{proof}
The 2-form $\xd\zh$ is the restriction to $M$ of
$$\xd\tilde\zh=\xd\, i_\n\zW=\Ll_\n\zW=\zW.$$
If $X\in(\ker(\xd\zh)\cap\ker(\zh))$, then $\zW(\n,X)=0$ and $\zW(\sT M,X)=0$, thus $X=0$.
\end{proof}
\no Here, $\Ll$ clearly denotes the Lie derivative. There are various generalizations of the above proposition, see for instance \cite{Grabowski:2004}. We have also a canonical reduction going in the reverse direction.
\begin{proposition}[contact-to-symplectic reduction]
Let $(M,\zh)$ be a regular contact manifold, and let $p:M\to N=M/\cR$ be the corresponding fibration. Then there is a unique symplectic form $\zw$ on $N$ such that $p^*(\zw)=\xd\zh$,
\end{proposition}
\begin{proof}
The kernel of the closed 2-form $\xd\zh$ on $M$ is spanned by $\cR$, so one can apply the standard symplectic reduction of presymplectic manifolds.
\end{proof}
\begin{definition} The procedure of passing from $(M,\zh)$ to $(N,\zw)$ we call the \emph{contact-to-symplectic reduction}, and the contact manifold $(M,\zh)$ -- a \emph{contactification} of the symplectic manifold $(N,\zw)$.
\end{definition}
\no The following example is well known in the literature (see e.g. \cite[Appendix 4]{Arnold:1989}).
\begin{example}\label{ex2}
Let $(N,\zw)$ be an exact symplectic manifold, $\zw=\xd\zvy$.  Then
$$\zh(x,t)=\zvy(x)+\xd t$$
is a contact form on $M=N\ti\R$ and $(M,\zh)$ is a contactification of $(N,\zw)$.
\end{example}
\no Finding contactifications of compact symplectic manifolds (which are never exact) is generally a more sophisticated task. Note also that contactifications are never unique, since any open submanifold $\cU\subset M$ of a contactification $(M,\zh)$ of $(N,\zw)$ which projects onto the whole $N$ is also a contactification of $(N,\zw)$ with the contact form $\zh\,\big|_\cU$. Particularly interesting are complete contactifications, e.g. compact contactifications of compact symplectic manifolds which cannot be obtained \emph{via} the above procedure.

\section{Regular contact manifolds with compact orbits}
Let us consider now a regular contact connected manifold $(M,\zh)$, so that $p:M\to N=M/\cR$ is a smooth fibration. The fibers of this fibration consist of orbits of the Reeb vector field $\cR$ (being closed submanifolds in the regular case) which are diffeomorphic either to circles (compact $\cR$-orbits) or to $\R$. On every compact orbit $\cO_x$, the flow generated by $\cR$ is periodic with the minimal period $\zr(x)$. Of course, if $M$ is compact, then all orbits, being closed, are circles automatically.

\mn Suppose for a moment that all $\cR$ orbits are circles. In this situation, Ehresmann's fibration theorem \cite{Ehresmann:1951}, stating that smooth fibrations $p:M\to N$ are locally trivial if only $p$ is a proper map (e.g., $M$ is compact), implies that our fibration by compact $\cR$-orbits is actually a locally trivial fibration. Indeed, this follows from the fact that every fiber has a tubular neighbourhood which is relatively compact.

\mn We would like to know whether the flow of $\cR$ is periodic as a whole.
To get the global periodicity, in \cite{Boothby:1958} the authors proved that the function $\zr(x)$ is smooth, and changed the contact form by multiplying $\zh$ by $1/\zr$. However, such an approach is a mistake, since the Reeb vector field of $\zr\cdot\zh$ is generally not $\cR/\zr$ as they claimed. Actually, the contact form $\zr\cdot\zh$ may even be no longer regular, as shown in Example \ref{e1}. But the situation is in fact much better, and we do not need this passage to an equivalent contact form, as shown in the following proposition.
\begin{proposition}\label{pmain} Suppose that, for a connected regular contact manifold $(M,\zh)$, the fibration $p:M\to N=M/\cR$ is actually a fiber bundle over $N$ with the typical fiber $S^1$. Then the flow generated by $\cR$ on $M$ is periodic with the minimal period $\zr$ which is common for all orbits, and defines a principal action of the group $S^1$, which turns $p:M\to N$ into an $S^1$-principal bundle with the principal connection $\zh$.

\mn Moreover, there exists a uniquely determined symplectic form $\zw$ on $N$ such that $p^*(\zw)=\xd\zh$, so $\zw$ represents the curvature of the connection $\zh$. The cohomology class of this symplectic form is $\Z_\zr$-integral, where $\Z_\zr=\zr\cdot\Z$, i.e., $[\zw/\zr]\in H^2(N,\Z)$.
\end{proposition}
\begin{proof}
Let us consider a local trivialization $M_U=p^{-1}(U)\simeq U\ti S^1$ of the fiber bundle $p:M\to N$, where $U$ is a connected open subset in $N$ with coordinates $(x^a)\in\R^{2n}$. It will be convenient to consider the standard covering of the circle
$$\R\ni \zt\mapsto [\zt]\in \R/\Z$$
and the corresponding covering
$$\zz:\tU=U\ti\R\to U\ti S^1.$$
It allows us to use coordinates $(x^a,\zt)$ on $\tU$ and to consider functions on $U\ti S^1$ as functions on $\tU$ which are 1-periodic with respect to $\zt$. The pull-back of $\zh$ to $\tU$ is a contact form $\th$ with the pull-back $\tR$ of $\cR$ as the Reeb vector field.

\m Let us write $\th$ in coordinates as
$$\th=g(x,\zt)\,\xd \zt+f_a(x,\zt)\,\xd x^a,$$
where $g(x,\zt)$ and $f_a(x,\zt)$ are 1-periodic in $\zt$. Since $\tR$ is tangent to the orbits,  $\tR=\pa_\zt/g$ and
\be\label{int}
\int_0^1g(x,\zt)\,\xd \zt=\zr(x)
\ee
is the minimal period of $\cR$ on $\cO_x$. We also have
$$
\xd\th=\frac{\pa g}{\pa x^a}(x,\zt)\,\xd x^a\we\xd \zt+\frac{\pa f_a}{\pa \zt}(x,\zt)\,\xd \zt\we\xd x^a
+\frac{\pa f_a}{\pa x^b}(x,\zt)\,\xd x^b\we\xd x^a.
$$
Since $i_\tR\xd\th=0$,
$$\Big(\frac{\pa g}{\pa x^a}(x,\zt)-\frac{\pa f_a}{\pa \zt}(x,\zt)\Big)\,\xd x^a=0,$$
so
$$\frac{\pa g}{\pa x^a}(x,\zt)=\frac{\pa f_a}{\pa \zt}(x,\zt)$$
for all $a$. Consequently (cf. (\ref{int})),
\beas &&\frac{\pa \zr}{\pa x^a}(x)=\frac{\pa}{\pa x^a}\Big(\int_0^1g(x,\zt)\,\xd \zt\Big) =\int_0^1\frac{\pa g}{\pa x^a}(x,\zt)\,\xd \zt\\
&&=\int_0^1\frac{\pa f_a}{\pa \zt}(x,\zt)\,\xd \zt=f_a(x,1)-f_a(x,0)=0,
\eeas
because $f_a(x,\zt)$ are 1-periodic with respect to $\zt$. Hence,
$$\frac{\pa \zr}{\pa x^a}(x)=0$$
for all $a$, thus $\zr(x)$ is constant on $U$. But $M$ is connected, so $\zr(x)$ is globally constant, $\zr(x)=\zr$. We have $\Exp(t\cR)=\id$
if and only if $t\in\Z_\zr$, where $\Z_\zr=\zr\cdot\Z$. In other words, $p:M\to N$ is a $\T_\hbar$-principal bundle in the terminology of \cite{Bates:1997}, where $\hbar=\zr/2\pi$.

\m Now, we will change the coordinates in $\tU$ into $(x^a,t)$, parametrizing fibers of $\tU$ by trajectories of the lifted Reeb vector field $\tR=\pa_\zt/g$,
$$\big(x,t(x,\zt)\big)=\Big(x,\int_0^\zt g(x,s)\xd s\Big).$$
A direct inspection shows that the diffeomorphism $(x,\zt)\mapsto(x,t(x,\zt))$ maps the vector field $\tR$ onto $\pa_t$.  Hence, in the coordinates $(x^a,t)$ our contact form reads
\be\label{zhU}\th=g(x,\zt)\xd\zt+f_a(x,\zt)\xd x^a=\xd t+h_a(x,t)\,\xd x^a\ee
for some functions $h_a$. By a direct check we get that $h_a(x,t)=f_a(x,0)$, so
$h_a(x,t)=h_a(x)$ does not depend on $t$ for all $a$, and we get
\be\label{zhU1}\xd\th=\frac{\pa h_a}{\pa x^b}(x)\,\xd x^b\we\xd x^a.\ee
Finally, since $\xd\zh$ depends only on coordinates $(x^a)$, it is the pull-back of a uniquely determined 2-form $\zw_U$ on $U$ having in coordinates $(x^a)$ formally the form (\ref{zhU1}). As $\xd\th$ is of rank $2n$, the form $\zw_U$ is symplectic. From the uniqueness of $\zw_U$ it follows that there is a globally defined symplectic form $\zw$ on $N$ such that $\zw\big|_U=\zw_U$, and $(M,\zh)$ is a contactification of $\zw$. Because for $\zvy_U=h_a(x)\,\xd x^a$ we have $\zw_U=\xd\zvy_U$,
each $\zw_U$ is exact, but clearly $\zw$ need not to be exact globally. Let us study this problem in more detail.

\mn To this end, let us choose a \v Cech cover $\{U_\za\}$ of $N$ (all intersections of the cover members are connected and contractible), so the $S^1$-bundles $p:M_\za=p^{-1}(U_\za)\to U_\za$ are trivial $S^1$-principal bundles, and equip $\tU_\za$ with local coordinates $(x^a_\za,t_\za)$ as above. The contact form $\zh$ in these coordinates reads (cf. (\ref{zhU}))
$$\zh_\za=\widetilde{\zh}_{U_\za}=\xd t_\za+h_a(x)\,\xd x^a_\za.$$
On the intersection $U_{\za\zb}=U_\za\cap U_\zb$ consider coordinates $(x^a,t_\za)$ and $(x^a,t_\zb)$, respectively, where $(x^a)$ are local coordinates on $U_{\za\zb}$, the same for $U_\za$ and $U_\zb$. Since the diffeomorphism $(x^a,t_\za)\mapsto(x^a,t_\zb)$ corresponds to an isomorphism of $S^1$-principal bundles, we have $t_\zb(x,t_\za)=t_\za+A_{\zb\za}(x)$ for some function $A_{\zb\za}:U_{\za\zb}\to\R$. Of course, $A_{\za\zb}=-A_{\zb\za}$. Since the shift of $t_\za$ by
$$T_{\zg\zb\za}=A_{\za\zg}+A_{\zg\zb}+A_{\zb\za}$$ induces the identity on $M_\za$, we have the cocycle condition
\be\label{cc}
T_{\zg\zb\za}=A_{\za\zg}+A_{\zg\zb}+A_{\zb\za}\in\Z_\zr=\zr\cdot\Z.\ee
The contact form $\zh_\zb$ on $U_{\za\zb}$ in coordinates $(x^a,t_\za)$ reads
$$\zh_\zb=\xd\big(t_\za+A_{\zb\za}\big)+\zvy_\za=\zh_\za+\xd A_{\zb\za}(x),$$
where $\zvy_\za=h_a(x)\,\xd x^a_\za$.
Consequently, on the intersection $U_{\za\zb}$ we have $\xd\zvy_\za=\xd\zvy_\zb=\zw$ and
$$\zvy_\zb-\zvy_\za=\xd A_{\zb\za}.$$
In view of (\ref{cc}) and the de Rham isomorphism, this means that the cohomology
class of the closed 2-form $\zw/\zr$ in $H^2(N;\R)$ lies, in fact, in the image of $H^2(N;\Z)$,
\be\label{ic}\big[\zw/\zr\big]\in H^2(N;\Z).\ee
Such closed 2-forms $\zw$ are called \emph{$\zr$-integral}, and it is known that they are characterized by the property $\int_\zS\zw\in\Z_\zr$ for each closed 2-dimensional surface $\zS$ in $N$.

\end{proof}
\no Note that we did not assume that $M$ is compact, as it is done in \cite{Boothby:1958}. This fact will be crucial for the next steps.
\begin{remark}
In the geometric quantization (see e.g. \cite{Bates:1997}), for $\zr$ is taken $2\pi\hbar$, where $\hbar$ is the Planck constant, and (\ref{ic}) is the well-known condition for the existence of a prequantum bundle on the symplectic manifold $(N,\zw)$.
This is because there is a one-to-one correspondence between $S^1$-principal bundles on $N$ and complex Hermitian line bundles $\C\hookrightarrow L\to N$.

Indeed, the multiplicative group $\C^\ti$ of nonzero complex numbers acts canonically on $L$, and the length $\|z\cdot v\|$ of the vector $z\cdot v\in L$ (with respect to the Hermitian metric) is $|z|\cdot\|v\|$. Identifying $S^1$ with complex numbers of modulus 1, it is easy to see that the set $M$ of length-1 vectors in $L$, $M=\{v\in L:\,\|v\|=1\}$, being preserved by the $S^1$-action, is automatically an $S^1$-principal bundle, the $S^1$-reduction of $L$.

Conversely, if $p:M\to N$ is an $S^1$-principal bundle, then the transition functions $$F_{\za\zb}:U_{\za\zb}=U_\za\cap U_\zb\to S^1$$
for a \v Cech cover $\{U_\za\}$ of $N$ realizing local trivializations, $p^{-1}(U_\za)\simeq U_\za\ti S^1$, form a \v Cech cocycle,
$$F_{\za\zg}\cdot F_{\zg\zb}\cdot F_{\zb\za}=1$$
on ${U_\za\cap U_\zb\cap U_\zg}$. Constructing a complex Hermitian line bundle from the local data $U_\za\ti\C$ (with the canonical Hermitian structure) and gluing them by
$$G_{\za\zb}:U_{\za\zb}\ti\C\to U_{\za\zb}\ti\C,\quad G_{\za\zb}(x,z)=(x,F_{\za\zb}(x)\cdot z),$$
we get a complex Hermitian line bundle whose reduction to the $S^1$-principal bundle gives back $M$.
\end{remark}

\section{The general case}
Let us assume again that $(M,\zh)$ is a regular contact manifold and $\cR$ is a complete vector field, thus its flow is global and generates a smooth action $(t,x)\mapsto\Exp(t\cR)(x)$ on $M$ of the group $(\R,+)$ of additive reals. \emph{A priori}, the dynamics of $\cR$ could contain both, compact and non-compact orbits. We will show that this is not possible. Note, however, that without the completeness assumption such examples do exist. For instance, one can take a regular contact compact manifold (like $S^3$ in Example \ref{e1}) and remove, say, one point from a fiber. Of course, after removing this point the Reeb vector field is no longer complete.

\subsection{No compact orbits}
Suppose first that all orbits are non-compact. In this case the $\R$-action induced by $\cR$ is free. Moreover, $p:M\to N=M/\cR$ is a fibration with fibers homeomorphic to $\R$, so automatically a fiber bundle (see e.g. \cite[Corollary 31]{Meigniez:2002}). It is easy to see that the free $\R$-action of the flow of $\cR$ which respects the fibers of this fiber bundle is automatically proper, so it turns this fiber bundle into an $\R$-principal bundle. Locally, $p^{-1}(U)=U\ti\R$, and using the flow of $\cR$ to parametrize the fibers, we get local coordinates $(x^a,t)$ on $U\ti\R$ such that $\cR=\pa_t$. In other words, the $\R$-action on $U\ti\R$ is $s.(x,t)=(x,t+s)$. This form of $\cR$ implies that $\zh$ can be locally written as
$$\zh=\xd t+f_a(x,t)\,\xd x^a.$$
Hence,
$$\xd\zh=\frac{\pa f_a}{\pa t}(x,t)\,\xd t\we\xd x^a
+\frac{\pa f_a}{\pa x^b}(x,t)\,\xd x^b\we\xd x^a,$$
and because $i_\cR\xd\zh=0$, we get $\frac{\pa f_a}{\pa t}(x,t)=0$ for all $a$. It follows that the functions $f_a$ do not depend on $t$, $f_a(x,t)=f_a(x)$, so $\zh=\xd t+\zvy_U$, where $\zvy_U=f_a(x)\xd x^a$ is the pull-back of a 1-form on $U$. Since
\be\label{symp}\xd\zh=\xd\zvy_U=\frac{\pa f_a}{\pa x^b}(x)\,\xd x^b\we\xd x^a\ee
is of rank $2n$, it is the pull-back of a uniquely determined symplectic form $\zw_U$ on $U$ which
in coordinates $(x^a)$ looks exactly like (\ref{symp}). Being uniquely determined by $\xd\zh$, the symplectic forms $\zw_U$ agree on the intersections $U_1\cap U_2$, so that there is a symplectic form $\zw$ on $N$ such that $\xd\zh=p^*(\zw)$. In other words, $(M,\zh)$ is a contactification of $(N,\zw)$ on which the $\R$-action generated by the flow of $\cR$ defines an $\R$-principal bundle structure. Moreover, $\zw$ represents the curvature of the principal connection $\zh$.

The fiber bundle $p:M\to N$ is clearly trivializable, as the fibers are contractible. Using a global section $\zs:N\to M$ to identify $N$ with a submanifold $\zs(N)$ in $M$, we can view $M\simeq N\ti\R$ as a trivial $\R$-principal bundle over $N$. Let $\zh_\zs$ be the contact form $\zh$ reduced to the horizontal submanifold $\zs(N)$. By the identification $N\simeq \zs(N)$ given by the section $\zs$ (or the projection $p$) we can view $\zh_\zs$ as a 1-form on $N$. Since the pull-back by $p$ of $\xd\zh_\zs$ is $\xd\zh$, we have $\xd\zh_\zs=\zw$, so the symplectic form $\zw$ is exact. Therefore, it is easy to see that $(M,\zh)$ is the standard contactification of the symplectic form $\zw=\xd\zh_\zs$ described in Example \ref{ex2}.

\subsection{There exists a compact orbit}
For a result being a variant of the celebrated \emph{Reeb local stability theorem} and describing the behavior of smooth fibrations near a compact fiber, we refer to Meigniez \cite[Lemma 22]{Meigniez:2002}. It simply says that, for a smooth fibration, every compact subset of every fibre has a product neighborhood. In our situation, this immediately implies that, if $x$ is a point in $N$ for which the orbit $\cO_x=p^{-1}(x)$ is compact, then there is a (connected) neighbourhood $U\subset N$ of $x$ such that $p$ is a fiber bundle when restricted to $p^{-1}(U)$. In other words, any compact orbit has a neighbourhood $M_U=p^{-1}(U)$ in which $p$ is a trivializable fiber bundle, $M_U\simeq U\ti S^1$, with the typical fiber $S^1$.

\mn It follows now from Proposition \ref{pmain} that in the open submanifold $M_U\subset M$ the Reeb vector field induces an $S^1$-principal action. Let us fix such $U$ and denote the corresponding period $\zr$. Let $$M_{\zr}=\{x\in N:\,\zr(x)=\zr\}$$
be the set of points of $N$ for which $\cO_x$ is a compact orbit with the minimal period $\zr$.
It is clear from what has been said that $M_{\zr}$ is open. We will show that it is also closed.

\mn Indeed, let $x_0\in N$ belong to the closure of $M_{\zr}$ and $y_0\in\cO_{x_0}$. The submersion $p:M\to N$ is an open map, so in a neighbourhood of $y_0$ there is a sequence of points $(y_n)$ such that $y_n\rightarrow y_0$ and $p(y_n)=x_n\in M_{\zr}$. Since $\cR$ is complete, its flow $\zf_t$ is globally defined for all $t\in\R$. We have then
$$y_n=\zf_\zr(y_n)\rightarrow\zf_\zr(y_0).$$
Hence, $\zf_\zr(y_0)=y_0$, so $\cO_{x_0}$ is also a periodic orbit with period $\zr$.
This is, in fact, the minimal period for $\cO_{x_0}$, since the minimal period (Proposition \ref{pmain}) is locally constant on periodic orbits.
For connected $N$ all this implies that, if there is one periodic orbit of $\cR$ with the minimal period $\zr$, then all orbits are periodic with the same minimal period $\zr$.

\subsection{The main result}
Summing up all our observations in the preceding sections, we can formulate the following general result.
\begin{theorem}
Let $(M,\zh)$ be a regular connected and complete contact manifold, and let $p:M\to N=M/\cR$ be the corresponding smooth fibration. Then, either the global flow of the corresponding Reeb vector field $\cR$ induces a free $S^1$-action with the minimal period $\zr$ that turns $p$ into an $S^1$-principal bundle, or it turns $p$ into an $\R$-principal bundle, which is clearly trivializable, $M\simeq N\ti\R$.
In both cases the contact form $\zh$ represents a principal connection of the principal bundle $p:M\to N$ and $(M,\zh)$ is a contactification of a uniquely determined symplectic structure $\zw$ on $N$, i.e., $p^*(\zw)=\xd\zh$.

\mn Moreover, in the case of the $S^1$-principal bundle, the symplectic form $\zw$ on $N$ is $\Z_\zr$-integral, where $\Z_\zr=\zr\cdot\Z$, i.e., $\big[\zw/\zr\big]\in H^2(N,\Z)$.
In the case of the $\R$-principal bundle, in turn, the symplectic form $\zw$ is exact, $\zw=\xd\zvy$, for a 1-form $\zvy$ and $(M\simeq N\ti\R,\zh)$ is the standard contactification of an exact symplectic manifold: $\zh(x,t)=\zvy(x)+\xd t$.
\end{theorem}
\no The following corollary about contactifications of symplectic manifolds is now obvious.
\begin{corollary}
Let $(N,\zw)$ be a compact connected symplectic manifold. Then for any contactification $(M,\zh)$ of $(N,\zw)$
with the complete Reeb vector field, the manifold $M$ is canonically an $S^1$-principal bundle (thus compact)
such that $\zh$ is invariant with respect to the $S^1$-action.
\end{corollary}
\section{Conclusions and outlook}
Our interests in the subject of this paper came from contact supergeometry \cite{Bruce:2017,Grabowski:2013}, contact mechanics \cite{Grabowska:2022,Grabowska:2022a}, and geometry of quantum states \cite{Grabowski:2005}.  In this paper, for a regular contact manifold $(M,\zh)$ we have shown that, under the assumption that the Reeb vector field $\cR$ is complete, the dynamics of $\cR$ is very rigid: either each orbit is compact and the flow of $\cR$ makes $M$ into an $S^1$-principal bundle, or all orbits are non-compact and the flow of $\cR$ makes $M$ into an $\R$-principal bundle. In both cases there is a unique symplectic form $\zw$ on the manifold $M/\cR$ of $\cR$-orbits such that $p^*(\zw)=\xd\zh$, where $p:M\to M/\cR$ is the canonical projection, so $(M,\zh)$ is a contactification of $(M/\cR,\zw)$. Moreover, $\zw$ satisfies an integrality condition of geometric quantization in the first case, and is an exact symplectic form in the other. Note that the corresponding $S^1$-principal bundles over $(M/\cR,\zw)$ are in a one-to-one correspondence with complex Hermitian line bundles over $M/\cR$, which makes a connection to geometric quantization and quantum physics.
We included also an example showing that even arbitrary small deformations of a contact form within its conformal class can result in a complete qualitative change of the dynamic of the corresponding Reeb vector field.

\mn Our results are more general than the ones announced in \cite{Boothby:1958} (e.g., we do not assume that $M$ is compact) and the way out of the problems concerning the reparametrization of $\cR$ is by showing that we actually do not need any reparametrization, because under completeness assumption all orbits of $\cR$ with one periodic orbit are automatically periodic with a common minimal period.

\mn A particularly interesting subject of further studies is explicit constructions of contactifications of coadjoint orbits of compact Lie groups by means of the well-known Marsden-Weinstein-Meyer symplectic reduction
\cite{Marsden:1982,Meyer:1973} or, more precisely, contact reductions in the spirit of \cite{Grabowska:2022a}.
Such contactifications are essentially known from quantum physics in the case of unitary groups. A celebrated example is the unit sphere in a Hilbert space being a contactification of the corresponding space of pure quantum states (complex projective space). This is related to another false statement we found in the literature, this time in paper \cite{Bravetti:2021}, that the Fubini-Study symplectic forms on complex projective spaces are exact, which is impossible for compact symplectic manifolds.

\section{Acknowledgements}
The authors thank Alan Weinstein for his clarifications concerning the Weinstein Conjecture.

\vskip.5cm
\noindent Katarzyna Grabowska\\\emph{Faculty of Physics,
University of Warsaw,}\\
{\small ul. Pasteura 5, 02-093 Warszawa, Poland} \\{\tt konieczn@fuw.edu.pl}\\
https://orcid.org/0000-0003-2805-1849\\

\noindent Janusz Grabowski\\\emph{Institute of Mathematics, Polish Academy of Sciences}\\{\small ul. \'Sniadeckich 8, 00-656 Warszawa,
Poland}\\{\tt jagrab@impan.pl}\\  https://orcid.org/0000-0001-8715-2370
\\
\end{document}